\documentclass[11pt]{article}

\usepackage[T1]{fontenc}
\usepackage[utf8]{inputenc}
\usepackage{lmodern}
\usepackage[a4paper,margin=2.6cm]{geometry}
\usepackage{amsmath,amssymb,amsthm,mathtools}
\usepackage{enumitem}
\usepackage{hyperref}
\usepackage{microtype}

\newtheorem{thm}{Theorem}[section]
\newtheorem{lem}[thm]{Lemma}
\newtheorem{prop}[thm]{Proposition}
\newtheorem{cor}[thm]{Corollary}
\newtheorem{open}[thm]{Open problem}
\theoremstyle{definition}
\newtheorem{defn}[thm]{Definition}
\theoremstyle{remark}
\newtheorem{rem}[thm]{Remark}

\newcommand{\R}{\mathbb{R}}

\newcommand{\Q}{\mathbb{Q}}

\newcommand{\id}{\operatorname{id}}

\title{N-ary quasi-arithmetic means and families without regularity}
\author{Gergely Kiss and Ekaterina Shulman}
\date{}

\begin{document}
\maketitle

\begin{abstract}
The classical theorems of Kolmogorov--Nagumo--de Finetti and of Aczel--Maksa characterize quasi-arithmetic means from two complementary directions: the former for compatible families of means satisfying the replacement axiom, and the latter for bisymmetric means of fixed arity. We refine both representation results by showing that the required continuity follows automatically. Our main result states that every reflexive, symmetric, bisymmetric and partially strictly increasing $n$-variable operation on a real interval is continuous and hence quasi-arithmetic. The proof is based on a recursive construction on $n$-adic rationals given by bisymmetry, and a dense-domain continuity argument. The same method also yields the regularity-free Kolmogorov--Nagumo--de Finetti theorem for compatible families of strictly increasing symmetric means.
\end{abstract}

\noindent\textbf{Keywords.} Bisymmetry; quasi-arithmetic mean; Kolmogorov--Nagumo--de Finetti mean; strict monotonicity; reflexivity; automatic continuity; $n$-ary operation.

\smallskip
\noindent\textbf{2020 Mathematics Subject Classification.} Primary 39B22; Secondary 26A15, 39B12.

\section{Introduction}
The theory of quasi-arithmetic means goes back to the classical works of Kolmogoroff,
Nagumo and de Finetti, who characterized families of means satisfying natural compatibility
and regularity assumptions (see \cite{deFinetti1931,Kolmogoroff1930,Nagumo1930}). In one
of its standard forms, a quasi-arithmetic mean generated by a continuous strictly monotone
bijection $\varphi:I\to J$, where $I$ and $J$ are real intervals, is given by
\[
M_n(x_1,\dots,x_n)=
\varphi^{-1}\!\left(\frac{\varphi(x_1)+\cdots+\varphi(x_n)}{n}\right),
\qquad n\ge2.
\]

Aczel's binary characterization gave a different and very influential approach. Instead of
starting from a whole compatible family $(M_n)_{n\ge1}$, one studies a single binary operation
$F:I^2\to I$ satisfying reflexivity, symmetry, strict monotonicity and the bisymmetry equation
\[
F(F(x,y),F(u,v))=F(F(x,u),F(y,v)).
\]
Under continuity, Aczel proved that such an operation is necessarily a quasi-arithmetic mean
\cite{Aczel1948} (see also \cite{AczelDhombres1989}). In the $n$-variable setting, the corresponding
continuous representation was obtained in the works of Maksa and of Maksa--Mokken--Munnich
\cite{Maksa1999,MaksaMokkenMunnich2000}.

From a different, more algebraic point of view, the same row--column interchange
law is usually called the medial or entropic law. In the quasigroup setting,
classical results of Toyoda, Murdoch and Bruck show that medial quasigroups admit
an affine representation over abelian groups; see
\cite{Bruck1944,Murdoch1941,Toyoda1941}. A related algebraic theory of abstract mean values was developed by Evans
\cite{Evans}. The resulting representation is not the quasi-arithmetic affine
representation studied here, but an algebraic one involving endomorphisms of
the underlying structure. Thus Evans' work provides a useful conceptual
predecessor: it shows, in an algebraic setting, how mean-value type identities
and suitable non-degeneracy assumptions can lead to a rigid structural
description. They do not, however, apply directly to our situation, since we do not
assume a quasigroup structure, cancellativity or solvability of equations. The
present paper treats the real-interval ordered setting, where reflexivity and
partial strict monotonicity replace these algebraic invertibility assumptions
and force the missing regularity.

Quasi-arithmetic means and bisymmetry have since appeared in several broader forms. Weighted
and generalized quasi-arithmetic means, Bajraktarevi\'c-type means, deviation means and
mean-type mappings have led to a rich functional-equation theory of means. In many of these
settings, bisymmetry or generalized bisymmetry expresses a compatibility of averaging
procedures, often in connection with Gauss composition, invariance equations or aggregation
theory (see, for instance, \cite{GrabischMarichalMesiarPap2009,MatkowskiPales2015}). Recent
work on generalized quasi-arithmetic means and related classes further illustrates the continuing
role of these questions in the theory of means \cite{PalesPasteczka2026,KissNagy2024}. Bisymmetry
has also been studied in more algebraic or order-theoretic settings, for example for quasitrivial
and order-preserving operations on chains \cite{DevilletKissMarichal2019}.

The continuity assumption is the central point of the present paper. In the binary symmetric
reflexive case, Burai, Kiss and Szokol proved that continuity is not needed: every reflexive,
symmetric, bisymmetric and partially strictly increasing binary operation on a non-degenerate real interval
is automatically continuous and hence quasi-arithmetic \cite{BuraiKissSzokol2021}. This raised
the natural $n$-variable question whether the same regularity improvement holds for symmetric
$n$-ary bisymmetric operations.

We answer this question affirmatively. The result is not a formal consequence
of the binary automatic-continuity theorem of Burai--Kiss--Szokol. Although
product arities admit reductions to lower arities (see Section \ref{sec:product}), the general $n$-variable
case requires a new structural description given in Section \ref{sec:Dn-construction}. In particular, one has to recover,
from symmetry and bisymmetry alone, a dense-domain averaging identity adapted to
the $n$-ary operation. This is analogous in spirit to the structural mechanism
behind Aczel's binary theorem, but it is not available from the classical
continuous representation theory, where continuity is assumed from the outset.

Thus the novelty is not the continuous representation itself, but the fact that
the required regularity and the corresponding quasi-arithmetic structure are
forced by symmetry, bisymmetry, reflexivity and order alone. In this sense, the
paper gives an automatic-continuity version of the
Maksa--Maksa--Mokken--Munnich representation theorem in the symmetric
$n$-variable case.

Our main theorem is the following.

\begin{thm}[Main theorem]\label{thm:main}
Let $I$ be a non-degenerate real interval, let $n\ge2$, and let
$F:I^n\to I$ be reflexive, symmetric, bisymmetric and partially strictly increasing.
Then there exist a non-degenerate real interval $J$ and a continuous strictly
monotone bijection $\varphi:I\to J$ such that
\[
F(x_1,\dots,x_n)=
\varphi^{-1}\!\left(\frac{\varphi(x_1)+\cdots+\varphi(x_n)}{n}\right)
\qquad (x_1,\dots,x_n\in I).
\]
In particular, $F$ is continuous.
\end{thm}

The proof is constructive and is modeled on the dyadic construction used in the binary case.
After reducing the problem to compact intervals $I=[a,b]$, we construct, for general $n$, a strictly increasing function on the set $D_n$ of $n$-adic rationals in $[0,1]$. The main algebraic step is to prove the identity
\[
F(f(t_1),\dots,f(t_n))=f\!\left(\frac{t_1+\cdots+t_n}{n}\right)
\qquad (t_i\in D_n),
\]
using symmetry and bisymmetry applied to suitably arranged arrays. Once this is known, a dense-domain gap argument forces
$f(D_n)$ to be dense in $I$, hence $f$ extends continuously to the whole interval and the
quasi-arithmetic representation follows.

A useful feature of the proof is the distinction between two dense sets. The full set $D_n$ is
needed for the complete $n$-variable averaging identity. On the other hand, the continuity
mechanism itself only uses comparisons of the two-point form
\[
\frac{(n-1)u+v}{n}.
\]
This leads naturally to the smaller set $E_n$ generated from $0$ and $1$ by
\[
\phi(u,v)=\frac{(n-1)u+v}{n}.
\]
We show that $E_n$ is always dense in $[0,1]$, while for $n\ge4$ it is a proper subset of $D_n$.
This separation clarifies why the dense-control part of the argument is simpler than the full
algebraic propagation of the averaging identity.

The same automatic-continuity mechanism also yields the following form of the
Kolmogorov--Nagumo--de Finetti theorem in which continuity is no longer assumed
a priori.

\begin{thm}[Kolmogorov--Nagumo--de Finetti theorem without continuity]
\label{thm:KND-without-continuity}
Let $I$ be a non-degenerate real interval, and let
$(M_n)_{n\ge1}$ be a family of symmetric means on $I$, with
$M_1=\id$. Assume that each $M_n$ is strictly increasing in each variable
and that the family satisfies the replacement axiom. Then there exist a
non-degenerate real interval $J$ and a continuous strictly monotone bijection
$\varphi:I\to J$ such that
\[
M_n(x_1,\dots,x_n)=
\varphi^{-1}\!\left(
\frac{\varphi(x_1)+\cdots+\varphi(x_n)}{n}
\right)
\]
for every $n\ge1$ and all $x_1,\dots,x_n\in I$. In particular, all the
means $M_n$ are automatically continuous.
\end{thm}

The structure of the paper reflects the development of the method. Section \ref{sec:pre}
collects the basic notions and recalls the continuous representation theorems
used later. We also explain why it is enough to prove the main theorem on compact
intervals. Section \ref{sec:ternary} treats the ternary case as a warm-up and presents, in a
concrete low-dimensional form, how the defining bisymmetry identity is applied
to a suitably arranged array of variables. Section \ref{sec:product} gives an alternative
reduction in product arities, showing in particular how the case $n=4$
(generally $n=2^k$) follows from the binary theorem. These two sections are
included not only as special cases, but also to make visible the mechanism that
will later be used in the general construction. In Section \ref{sec:En-Dn} we introduce two
dense sets $D_n$ and $E_n$, separating the full algebraic domain $D_n$
needed for the averaging identity from the smaller control set $E_n$ needed
for continuity. Section \ref{sec:Dn-construction} contains the general $D_n$-construction and proves
the averaging identity on $D_n$. In Section \ref{sec:dense-lemma} we prove the dense-domain
continuity lemma, and in Section \ref{sec:proof-main} we complete the proof of
Theorem~\ref{thm:main}. Finally, Section \ref{sec:kolmogorov} proves
Theorem~\ref{thm:KND-without-continuity}, the continuity-free
Kolmogorov--Nagumo--de Finetti theorem for compatible families of strictly
increasing symmetric means. Section \ref{sec:open} discusses the remaining non-symmetric
problem.

The paper concludes with the principal remaining problem: whether symmetry can be omitted.
More precisely, it is open whether reflexivity, bisymmetry and partial strict monotonicity alone
force continuity. Thus the symmetric case is settled by the present paper, while
the genuinely non-symmetric case remains open. This is the natural boundary of
the present method. Without reflexivity, continuity is known to fail. In \cite{Kiss2026},
discontinuous bisymmetric strictly increasing operations are constructed in the
binary case and in arbitrary arity; the construction also yields genuinely
multivariate examples.
\section{Preliminaries}
\label{sec:pre}

Throughout the paper $I$ denotes a non-degenerate real interval. Later we
shall restrict our attention to compact intervals $I=[a,b]$, using
Lemma~\ref{lem:compact-reduction}.

\begin{defn}[Reflexivity]
A map $F:I^n\to I$ is called \emph{reflexive} if
\[
F(x,\dots,x)=x\qquad (x\in I).
\]
\end{defn}

\begin{defn}[Symmetry]
A map $F:I^n\to I$ is called \emph{symmetric} if
\[
F(x_1,\dots,x_n)=F(x_{\sigma(1)},\dots,x_{\sigma(n)})
\]
for all $x_1,\dots,x_n\in I$ and all permutations $\sigma\in S_n$.
\end{defn}

\begin{defn}[Bisymmetry]
A map $F:I^n\to I$ is called \emph{bisymmetric} if
\begin{align*}
&F\bigl(F(x_{1,1},\dots,x_{1,n}),F(x_{2,1},\dots,x_{2,n}),\dots,F(x_{n,1},\dots,x_{n,n})\bigr)\\
&\quad=
F\bigl(F(x_{1,1},\dots,x_{n,1}),F(x_{1,2},\dots,x_{n,2}),\dots,F(x_{1,n},\dots,x_{n,n})\bigr)
\end{align*}
for all $x_{i,j}\in I$.
\end{defn}

\begin{defn}[Partial strict increase]
A map $F:I^n\to\R$ is \emph{partially strictly increasing} if for every $j\in\{1,\dots,n\}$, the one-variable section
\[
t\longmapsto F(x_1,\dots,x_{j-1},t,x_{j+1},\dots,x_n)
\]
is strictly increasing on $I$ whenever the remaining variables are fixed.
\end{defn}

\begin{defn}[Mean]
A map $F:I^n\to I$ is a \emph{mean} if
\[
\min_i x_i\le F(x_1,\dots,x_n)\le \max_i x_i
\qquad (x_i\in I).
\]
It is a \emph{strict mean} if both inequalities are strict whenever the variables are not all equal.
\end{defn}

\begin{lem}\label{lem:strict-mean}
If $F:I^n\to I$ is reflexive and partially strictly increasing, then $F$ is a strict mean.
\end{lem}

\begin{proof}
Let $m=\min_i x_i$ and $M=\max_i x_i$. By partial monotonicity and reflexivity,
\[
m=F(m,\dots,m)\le F(x_1,\dots,x_n)\le F(M,\dots,M)=M.
\]
If the variables are not all equal, then at least one coordinate is strictly larger than $m$, and at least one coordinate is strictly smaller than $M$. Applying partial strict increase coordinate by coordinate gives
\[
m<F(x_1,\dots,x_n)<M.
\]
\end{proof}

\begin{defn}[Replacement axiom]\label{def:replacement}
Let $I$ be a real interval, and let
\[
M_n:I^n\to I\qquad(n\ge1)
\]
be a family of means with $M_1=\id$. We say that the family satisfies the
\emph{replacement axiom}\footnote{In parts of the early literature this condition
was called associativity. We use the term replacement axiom in order to avoid
confusion with the semigroup-theoretic meaning of associativity.} if, for all
$k,m\ge1$,
\begin{multline}\label{eq:replacement-axiom}
M_{k+m}(x_1,\dots,x_k,y_1,\dots,y_m)\\
=
M_{k+m}\bigl(
\underbrace{M_k(x_1,\dots,x_k),\dots,M_k(x_1,\dots,x_k)}_{k},
y_1,\dots,y_m
\bigr)
\end{multline}
for all $x_1,\dots,x_k,y_1,\dots,y_m\in I$. If the means $M_n$ are symmetric,
then the same replacement may be applied to any block of variables.
\end{defn}

We recall the classical continuous theorems in the form needed below.

\begin{thm}[Kolmogoroff--Nagumo--de Finetti]\label{thm:KND}
Let $I$ be a real interval, and let $(M_n)_{n\ge1}$ be a family of symmetric
means on $I$, with $M_1=\id$. Assume that each $M_n$ is continuous and
strictly increasing in each variable, and that the family satisfies the
replacement axiom. Then there exist a non-degenerate real interval $J$ and a
continuous strictly monotone bijection $\varphi:I\to J$ such that \[
M_n(x_1,\dots,x_n)=
\varphi^{-1}\!\left(\frac{\varphi(x_1)+\cdots+\varphi(x_n)}{n}\right)
\]
for every $n\ge1$ and all $x_1,\dots,x_n\in I$.
\end{thm}

\begin{thm}[Aczel, binary case]\label{thm:Aczel}
Let $I$ be a real interval. A map $F:I^2\to I$ is continuous, reflexive, symmetric, bisymmetric and partially
strictly increasing if and only if there
exist a non-degenerate real interval $J$ and a continuous strictly monotone
bijection $\varphi:I\to J$ such that
\[
F(x,y)=\varphi^{-1}\!\left(\frac{\varphi(x)+\varphi(y)}{2}\right)
\qquad (x,y\in I).
\]
\end{thm}

\begin{thm}[Maksa; Maksa--Mokken--Munnich]\label{thm:Maksa}
Let $I$ be a real interval and let $F:I^n\to I$ be continuous, reflexive, symmetric, bisymmetric and partially
strictly increasing. Then there exist a
non-degenerate real interval $J$ and a continuous strictly monotone bijection
$\varphi:I\to J$ such that
\[
F(x_1,\dots,x_n)=\varphi^{-1}\!\left(\frac{\varphi(x_1)+\cdots+\varphi(x_n)}{n}\right)
\qquad (x_i\in I).
\]
\end{thm}

\begin{thm}[Burai--Kiss--Szokol]\label{thm:BKSz}
Let $I$ be a non-degenerate real interval. If $F:I^2\to I$ is reflexive, symmetric, bisymmetric
and partially strictly increasing, then $F$ is continuous. Consequently $F$ is a quasi-arithmetic mean.
\end{thm}

Finally, we record the reduction, which allows us to carry out the main construction
only on compact intervals.

\begin{lem}[Reduction to compact intervals]\label{lem:compact-reduction}
Assume that, for every non-degenerate compact interval $K$, every reflexive,
symmetric, bisymmetric and partially strictly increasing map $G:K^n\to K$ is
continuous. Then Theorem~\ref{thm:main} holds for arbitrary non-degenerate real
intervals.
\end{lem}

\begin{proof}
Let $I$ be a non-degenerate real interval and let $F:I^n\to I$ be reflexive,
symmetric, bisymmetric and partially strictly increasing. By
Lemma~\ref{lem:strict-mean}, $F$ is a strict mean. Hence, for every
non-degenerate compact subinterval $K\subset I$, the restriction $F|_{K^n}$
maps $K^n$ into $K$ and satisfies the same assumptions.

By the assumed compact case, $F|_{K^n}$ is continuous for every such compact
subinterval $K$. We prove that $F$ is continuous on $I^n$.

Let $x=(x_1,\dots,x_n)\in I^n$. Choose a non-degenerate compact subinterval
$K\subset I$ such that each $x_i$ belongs to the interior of $K$ relative
to $I$. This is possible as follows. If $I$ contains a left endpoint and
some $x_i$ is equal to it, then we choose $K$ with the same left endpoint.
Similarly, if $I$ contains a right endpoint and some $x_i$ is equal to it,
then we choose $K$ with the same right endpoint. In all other directions we
choose the endpoints of $K$ strictly beyond the finitely many points
$x_1,\dots,x_n$, inside $I$.

Then $K^n$ contains a neighbourhood of $x$ in the relative topology of
$I^n$. Since $F$ agrees on this neighbourhood with the continuous map
$F|_{K^n}$, the map $F$ is continuous at $x$. As $x\in I^n$ was arbitrary,
$F$ is continuous on $I^n$.

The continuous representation theorem of Maksa and of Maksa--Mokken--Munnich,
Theorem~\ref{thm:Maksa}, now applies. Therefore there exist a non-degenerate real
interval $J$ and a continuous strictly monotone bijection $\varphi:I\to J$
such that
\[
F(x_1,\dots,x_n)=
\varphi^{-1}\!\left(\frac{\varphi(x_1)+\cdots+\varphi(x_n)}{n}\right)
\qquad (x_1,\dots,x_n\in I).
\]
Thus Theorem~\ref{thm:main} follows from the compact case.
\end{proof}

\begin{rem}
In the last step of the proof of Lemma~\ref{lem:compact-reduction} we invoked
the continuous representation theorem of Maksa and of Maksa--Mokken--Munnich.
This is the shortest way to pass from local continuity to the global
quasi-arithmetic representation. Alternatively, one could avoid this invocation
by using the standard affine uniqueness of quasi-arithmetic generators: if two
continuous strictly monotone generators represent the same quasi-arithmetic mean
on an interval, then they differ by a non-zero affine change of scale. Hence the
local generators obtained on compact subintervals can be normalized consistently
and pasted together. We shall not need this alternative argument.
\end{rem}

\medskip
\noindent\textbf{Standing assumption for the rest of the paper.}
By Lemma~\ref{lem:compact-reduction}, it is enough to establish automatic
continuity on compact intervals. Therefore, from this point on, all constructive
arguments will be carried out under the standing assumption that
\[
I=[a,b]
\]
is a non-degenerate compact interval.
\medskip

\section{A warm-up: the ternary case}
\label{sec:ternary}

Before proving the general theorem, we discuss the ternary case separately. This special case already contains
the main combinatorial idea: the coding function is constructed on triadic rationals, and bisymmetry is used
as a row--column interchange.

Let $F:[a,b]^3\to[a,b]$ be reflexive, symmetric, bisymmetric and partially strictly increasing. Set
\[
D_3=\left\{\frac{m}{3^r}: r  \ge 0,\ 0\le m\le 3^r\right\},
\qquad
D_{3,r}=\left\{\frac{m}{3^r}:0\le m\le 3^r\right\}.
\]
We define $f:D_3\to[a,b]$ recursively. Let $f(0)=a$ and $f(1)=b$. (This convention we use throughout the construction.) If $f$ is already defined on $D_{3,r}$,
define, for $s\in\{0,\dots,3^r-1\}$ and $\delta\in\{0,1,2\}$,
\begin{equation}\label{eq:ternary-def}
f\!\left(\frac{3s+\delta}{3^{r+1}}\right)=
F\Bigl(
\underbrace{f\!\left(\frac{s}{3^r}\right),\dots,f\!\left(\frac{s}{3^r}\right)}_{3-\delta},
\underbrace{f\!\left(\frac{s+1}{3^r}\right),\dots,f\!\left(\frac{s+1}{3^r}\right)}_{\delta}
\Bigr).
\end{equation}
For $\delta=0$ this agrees with the previous value by reflexivity.

\begin{prop}\label{prop:ternary-monotone}
The function $f:D_3\to[a,b]$ defined by \eqref{eq:ternary-def} is strictly increasing.
\end{prop}

\begin{proof}
The proof is by induction on $r$. It is clear on $D_{3,0}=\{0,1\}$. Assume strict increase on $D_{3,r}$.
Let
\[
x=\frac{3s+\delta}{3^{r+1}},\qquad y=\frac{3t+\epsilon}{3^{r+1}},\qquad x<y.
\]
If $s=t$, then $\delta<\epsilon$, and the defining expression for $f(y)$ is obtained from that for $f(x)$
by replacing at least one copy of $f(s/3^r)$ by the larger value $f((s+1)/3^r)$. Strict increase of $F$ gives
$f(x)<f(y)$. If $s<t$, then
\[
\frac{s}{3^r}<\frac{t}{3^r},\qquad \frac{s+1}{3^r}\le \frac{t}{3^r},
\]
and the induction hypothesis again gives coordinatewise domination, with at least one strict inequality.
Thus $f(x)<f(y)$.
\end{proof}

\begin{thm}\label{thm:ternary-identity}
For all $x,y,z\in D_3$,
\[
F(f(x),f(y),f(z))=f\!\left(\frac{x+y+z}{3}\right).
\]
\end{thm}

\begin{proof}
We prove the identity on $D_{3,r}$ by induction on $r$. The case $r=0$ follows directly from the definition.
Assume the identity on $D_{3,r}$ and write
\[
x=\frac{3k+\alpha}{3^{r+1}},\qquad
 y=\frac{3p+\beta}{3^{r+1}},\qquad
 z=\frac{3q+\gamma}{3^{r+1}},
\]
where $\alpha,\beta,\gamma\in\{0,1,2\}$. Expanding $f(x),f(y),f(z)$ by \eqref{eq:ternary-def}, we obtain
an outer $F$ applied to three inner $F$-values, hence to the row-values of a $3\times3$ matrix. By symmetry,
we may arrange the $+1$ entries in the three rows so that if
\[
\alpha+\beta+\gamma=3u+\delta,
\qquad \delta\in\{0,1,2\},
\]
then each column contains either $u$ or $u+1$ such entries, with exactly $\delta$ columns of the latter type.
Bisymmetry changes the row evaluation into column evaluation. By the induction hypothesis, the column
values are
\[
f\!\left(\frac{k+p+q+u}{3^{r+1}}\right)
\quad\text{or}\quad
f\!\left(\frac{k+p+q+u+1}{3^{r+1}}\right),
\]
with the second value occurring exactly $\delta$ times. A final use of the recursive definition gives
\[
F(f(x),f(y),f(z))=
f\!\left(\frac{3(k+p+q+u)+\delta}{3^{r+2}}\right)
=f\!\left(\frac{x+y+z}{3}\right).
\]
\end{proof}

The passage from this identity on $D_3$ to continuity is a special case of the dense-domain lemma proved in
Section~\ref{sec:dense-lemma}. Thus the ternary case is fully covered by the general mechanism below.

\section{Alternative reductions in product arities}
\label{sec:product}

The proof of the main theorem will not rely on this section. Nevertheless,
product arities admit a useful alternative reduction. This explains, for
instance, why the case $n=4$ follows from the binary theorem.

\subsection{The case \texorpdfstring{$n=4$}{n=4}}

\begin{thm}\label{thm:n4}
Let $I=[a,b]$ be a non-degenerate compact interval and let
$F:I^4\to I$ be reflexive, symmetric, bisymmetric and partially strictly
increasing. Then $F$ is continuous and is a $4$-variable quasi-arithmetic
mean.
\end{thm}

\begin{proof}
Define
\[
H(x,y):=F(x,x,y,y)\qquad (x,y\in I).
\]
Then $H$ is reflexive and symmetric. It is also partially strictly increasing:
if $x<x'$, then
\[
F(x,x,y,y)<F(x',x,y,y)<F(x',x',y,y),
\]
and hence $H(x,y)<H(x',y)$. The second variable is treated similarly.

We first prove the key identity
\begin{equation}\label{eq:n4-key}
H(H(x,y),H(u,v))=F(x,y,u,v)
\qquad (x,y,u,v\in I).
\end{equation}
Indeed, by definition,
\[
H(H(x,y),H(u,v))
=
F(H(x,y),H(x,y),H(u,v),H(u,v)).
\]
Using symmetry of $F$ in the inner values $H(x,y)$ and $H(u,v)$, this can
be written as
\[
F\Bigl(
F(x,x,y,y),
F(y,y,x,x),
F(u,u,v,v),
F(v,v,u,u)
\Bigr).
\]
Applying the bisymmetry of $F$ to this expression gives
\[
F\Bigl(
F(x,y,u,v),
F(x,y,u,v),
F(y,x,v,u),
F(y,x,v,u)
\Bigr).
\]
By symmetry of $F$, all four inner values are equal to $F(x,y,u,v)$. Hence
reflexivity gives \eqref{eq:n4-key}.

Now \eqref{eq:n4-key} immediately implies that $H$ is bisymmetric. Indeed,
\[
H(H(x,y),H(u,v))=F(x,y,u,v),
\]
while
\[
H(H(x,u),H(y,v))=F(x,u,y,v).
\]
The two right-hand sides are equal by symmetry of $F$. Therefore
\[
H(H(x,y),H(u,v))=H(H(x,u),H(y,v)).
\]

Thus $H$ is reflexive, symmetric, bisymmetric and partially strictly
increasing. By Theorem~\ref{thm:BKSz}, $H$ is continuous and is a binary
quasi-arithmetic mean. Hence there exist a non-degenerate real interval $J$
and a continuous strictly monotone bijection $\varphi:I\to J$ such that
\[
H(x,y)=
\varphi^{-1}\!\left(\frac{\varphi(x)+\varphi(y)}{2}\right)
\qquad (x,y\in I).
\]
Using \eqref{eq:n4-key} once more, we obtain
\begin{align*}
F(x,y,u,v)
&=H(H(x,y),H(u,v))\\
&=
\varphi^{-1}\!\left(
\frac{\varphi(H(x,y))+\varphi(H(u,v))}{2}
\right)\\
&=
\varphi^{-1}\!\left(
\frac{
\frac{\varphi(x)+\varphi(y)}{2}
+
\frac{\varphi(u)+\varphi(v)}{2}
}{2}
\right)\\
&=
\varphi^{-1}\!\left(
\frac{\varphi(x)+\varphi(y)+\varphi(u)+\varphi(v)}{4}
\right).
\end{align*}
In particular, $F$ is continuous and is a $4$-variable quasi-arithmetic mean.
\end{proof}

\subsection{A product-arity reduction}

The preceding argument gives the product-arity reduction in its most transparent
form for $n=4$. The same idea works for arbitrary product arities. Since this
section is only complementary, and since Sections~\ref{sec:Dn-construction}--%
\ref{sec:proof-main} will prove the main theorem directly for all arities, we
give a compact proof of the general reduction, keeping only the indexing details
that are needed to make the construction unambiguous.

\begin{prop}[Product-arity reduction]\label{prop:product-arity}
Let $k,\ell\ge2$ and set $n=k\ell$. Assume that the automatic-continuity
conclusion of Theorem~\ref{thm:main} is already known for arities $k$ and
$\ell$. Then it is also true for arity $n$.
\end{prop}

\begin{proof}
Let $I=[a,b]$, and let $F:I^n\to I$ be reflexive, symmetric, bisymmetric and
partially strictly increasing. Define
\begin{align}
H(x_1,\dots,x_k)
&:=F(\underbrace{x_1,\dots,x_1}_{\ell},\dots,
\underbrace{x_k,\dots,x_k}_{\ell}),\label{eq:def-H-product}\\
G(y_1,\dots,y_\ell)
&:=F(\underbrace{y_1,\dots,y_1}_{k},\dots,
\underbrace{y_\ell,\dots,y_\ell}_{k}).\label{eq:def-G-product}
\end{align}
Then $H$ and $G$ are reflexive, symmetric and partially strictly increasing.

We show that $H$ is bisymmetric; the proof for $G$ is the same with $k$
and $\ell$ interchanged. Let $x_{ij}\in I$, $1\le i,j\le k$. Consider the
$n\times n$ array whose rows are indexed by $(i,s)$, $1\le i\le k$,
$1\le s\le\ell$, and whose columns are indexed by $(j,t)$,
$1\le j\le k$, $1\le t\le\ell$, with entry
\[
z_{(i,s),(j,t)}=x_{ij}.
\]
For fixed $(i,s)$, the row value is $H(x_{i1},\dots,x_{ik})$, and this value
appears $\ell$ times as $s$ varies. Hence the row evaluation of the array is
\[
H\bigl(H(x_{11},\dots,x_{1k}),\dots,H(x_{k1},\dots,x_{kk})\bigr).
\]
For fixed $(j,t)$, the column value is $H(x_{1j},\dots,x_{kj})$, again
appearing $\ell$ times as $t$ varies. Hence the column evaluation is
\[
H\bigl(H(x_{11},\dots,x_{k1}),\dots,H(x_{1k},\dots,x_{kk})\bigr).
\]
Bisymmetry of $F$ makes these two evaluations equal, so $H$ is bisymmetric.

Thus $H$ and $G$ satisfy the hypotheses in arities $k$ and $\ell$,
respectively. By the assumed cases, both $H$ and $G$ are continuous.

We now prove the block-composition identity
\begin{equation}\label{eq:product-composition}
F((a_{ij})_{1\le i\le k,\ 1\le j\le\ell})
=
H\bigl(G(a_{11},\dots,a_{1\ell}),\dots,G(a_{k1},\dots,a_{k\ell})\bigr),
\end{equation}
where the left-hand side means $F$ applied to the $k\ell$ entries $a_{ij}$,
in any fixed order.

Set
\[
R_i:=G(a_{i1},\dots,a_{i\ell})\qquad(i=1,\dots,k).
\]
Then the right-hand side of \eqref{eq:product-composition} is
\[
H(R_1,\dots,R_k)
=
F(\underbrace{R_1,\dots,R_1}_{\ell},\dots,
\underbrace{R_k,\dots,R_k}_{\ell}).
\]
We realize this expression as the row evaluation of an $n\times n$ array.
Index the rows by pairs $(i,t)$, where $1\le i\le k$ and $1\le t\le\ell$,
and the columns by pairs $(r,s)$, where $1\le r\le\ell$ and $1\le s\le k$.
In row $(i,t)$ and column $(r,s)$, place
\[
a_{i,\,1+((r+t-2)\bmod \ell)}.
\]
For fixed $i,t$, the row contains $k$ copies of each of
$a_{i1},\dots,a_{i\ell}$, hence its row value is $R_i$. Therefore the row
evaluation is $H(R_1,\dots,R_k)$.

For fixed $r,s$, as $t$ runs through $1,\dots,\ell$, the index
\[
1+((r+t-2)\bmod \ell)
\]
runs through $1,\dots,\ell$. Hence the column $(r,s)$ contains all entries
$a_{ij}$, $1\le i\le k$, $1\le j\le\ell$, each exactly once. Thus every
column value is, by symmetry of $F$,
\[
F((a_{ij})_{i,j}).
\]
Applying bisymmetry of $F$ to this array, the row evaluation is equal to $F$
applied to $n$ identical copies of this common column value. Reflexivity gives
\eqref{eq:product-composition}.

Since $H$ and $G$ are continuous, \eqref{eq:product-composition} shows that
$F$ is continuous. The quasi-arithmetic representation of $F$ then follows
from the continuous $n$-variable theorem, Theorem~\ref{thm:Maksa}.
\end{proof}

\begin{cor}\label{cor:powers}
If Theorem~\ref{thm:main} is known for an arity $m\ge2$, then it is known for
all arities $m^r$, $r\ge1$. In particular, together with the binary theorem
and the ternary argument above, Proposition~\ref{prop:product-arity} explains
the product arities $2^k3^\ell$, where $k,\ell\ge0$ and $k+\ell\ge1$.
\end{cor}

\begin{proof}
Apply Proposition~\ref{prop:product-arity} inductively.
\end{proof}

\section{Two dense sets: the control set \texorpdfstring{$E_n$}{En} and the full set \texorpdfstring{$D_n$}{Dn}}
\label{sec:En-Dn}

The proof below involves two different dense subsets of $[0,1]$. The first one
is the full $n$-adic set $D_n$. It is the natural domain for the algebraic
part of the proof, because it is closed under the arithmetic average of $n$
independent $n$-adic inputs. This is the set on which we shall prove the full
identity
\[
F(f(t_1),\dots,f(t_n))
=
f\!\left(\frac{t_1+\cdots+t_n}{n}\right).
\]

The second set, denoted by $E_n$, is smaller and is generated only by the
two-point operation
\[
\phi(u,v)=\frac{(n-1)u+v}{n}.
\]
It is the natural control set for the dense-domain gap argument: the decisive
comparisons in the continuity proof use precisely averages of this form. Thus
$D_n$ belongs to the algebraic propagation of the full averaging identity,
whereas $E_n$ reflects the smaller amount of density needed for the continuity
mechanism.

For $n=2$ and $n=3$ these two sets coincide. Starting from $n=4$, however,
they are genuinely different: $E_n$ is still dense in $[0,1]$, but it is a
proper subset of $D_n$.

\begin{defn}\label{def:Dn}
Let
\[
D_n:=\left\{\frac{r}{n^m}:m\ge0,\ 0\le r\le n^m\right\},
\qquad
D_{n,m}:=\left\{\frac{k}{n^m}:0\le k\le n^m\right\}.
\]
Thus $D_n=\bigcup_{m\ge0}D_{n,m}$.
\end{defn}

\begin{lem}\label{lem:Dn-dense}
The set $D_n$ is dense in $[0,1]$ and is closed under $n$-fold arithmetic
averaging.
\end{lem}

\begin{proof}
Density is immediate from the mesh size $n^{-m}$ of $D_{n,m}$. If
$x_i\in D_n$, choose a common level $m$ with $x_i=r_i/n^m$. Then
\[
\frac{x_1+\cdots+x_n}{n}=\frac{r_1+\cdots+r_n}{n^{m+1}}\in D_n.
\]
\end{proof}

\begin{defn}\label{def:En}
Define
\[
E_n^{(0)}:=\{0,1\},
\qquad
E_n^{(m+1)}:=
\left\{\frac{(n-1)u+v}{n}:u,v\in E_n^{(m)}\right\},
\]
and set
\[
E_n:=\bigcup_{m\ge0}E_n^{(m)}.
\]
\end{defn}

\begin{lem}\label{lem:En-nesting}
For every $m\ge0$, one has $E_n^{(m)}\subseteq E_n^{(m+1)}$. Moreover,
$E_n\subseteq D_n$.
\end{lem}

\begin{proof}
The nesting follows from
\[
x=\frac{(n-1)x+x}{n}.
\]
For the inclusion, use induction on $m$. The statement is clear for $m=0$.
If $u=a/n^m$ and $v=b/n^m$, then
\[
\frac{(n-1)u+v}{n}=\frac{(n-1)a+b}{n^{m+1}},
\]
so $E_n^{(m+1)}\subseteq D_{n,m+1}$ whenever $E_n^{(m)}\subseteq D_{n,m}$.
\end{proof}

\begin{lem}\label{lem:E2E3}
For $n=2$ and $n=3$, one has $E_n=D_n$.
\end{lem}

\begin{proof}
The inclusion $E_n\subseteq D_n$ follows from Lemma~\ref{lem:En-nesting}. We
prove the reverse inclusion separately for $n=2$ and $n=3$.

For $n=2$, we prove by induction that $D_{2,m}\subseteq E_2^{(m)}$ for all
$m\ge0$. The case $m=0$ is clear. Assume $D_{2,m}\subseteq E_2^{(m)}$, and
take
\[
\frac{2q+\delta}{2^{m+1}}\in D_{2,m+1},
\qquad
\delta\in\{0,1\}.
\]
If $\delta=0$, then
\[
\frac{2q}{2^{m+1}}
=
\frac{q/2^m+q/2^m}{2}.
\]
If $\delta=1$, then
\[
\frac{2q+1}{2^{m+1}}
=
\frac{q/2^m+(q+1)/2^m}{2}.
\]
In both cases the two entries on the right belong to $D_{2,m}$, hence to
$E_2^{(m)}$. Thus $D_{2,m+1}\subseteq E_2^{(m+1)}$.

For $n=3$, we prove by induction that $D_{3,m}\subseteq E_3^{(m)}$ for all
$m\ge0$. Again the case $m=0$ is clear. Assume
$D_{3,m}\subseteq E_3^{(m)}$, and take
\[
\frac{3q+\delta}{3^{m+1}}\in D_{3,m+1},
\qquad
\delta\in\{0,1,2\}.
\]
If $\delta=0$, then
\[
\frac{3q}{3^{m+1}}
=
\frac{2(q/3^m)+q/3^m}{3}.
\]
If $\delta=1$, then
\[
\frac{3q+1}{3^{m+1}}
=
\frac{2(q/3^m)+(q+1)/3^m}{3}.
\]
If $\delta=2$, then
\[
\frac{3q+2}{3^{m+1}}
=
\frac{2((q+1)/3^m)+q/3^m}{3}.
\]
In each case the two entries on the right belong to $D_{3,m}$, hence to
$E_3^{(m)}$. Thus $D_{3,m+1}\subseteq E_3^{(m+1)}$.
\end{proof}

\begin{lem}[Density of $E_n$]\label{lem:En-dense}
The set $E_n$ is dense in $[0,1]$.
\end{lem}

\begin{proof}
Fix $t\in[0,1]$. For each $m$, let
\[
L_m:=\max\{x\in E_n^{(m)}:x\le t\},
\qquad
R_m:=\min\{x\in E_n^{(m)}:x\ge t\}.
\]
If $t\in E_n^{(m)}$ for some $m$, there is nothing to prove. Otherwise
$L_m<t<R_m$ for all $m$. Set
\[
A_m=\frac{(n-1)L_m+R_m}{n},
\qquad
B_m=\frac{L_m+(n-1)R_m}{n}.
\]
Then $A_m,B_m\in E_n^{(m+1)}$ and
\[
L_m<A_m\le B_m<R_m.
\]
In the three cases $t\le A_m$, $A_m\le t\le B_m$, and $B_m\le t$, the
interval between the new nearest left and right points has length at most
\[
\frac{n-1}{n}(R_m-L_m).
\]
Hence
\[
R_m-L_m\le\left(\frac{n-1}{n}\right)^m(R_0-L_0)\to0.
\]
Thus points of $E_n$ approximate $t$ arbitrarily well.
\end{proof}

\begin{lem}[Residue invariant]\label{lem:residue-invariant}
If $z\in E_n^{(m)}$ and $z=k/n^m$, then
\[
k\equiv 0 \quad\text{or}\quad 1 \pmod{n-1}.
\]
\end{lem}

\begin{proof}
We argue by induction on $m$. The case $m=0$ is clear. Let
$z\in E_n^{(m+1)}$. Then
\[
z=\frac{(n-1)u+v}{n}
\]
with $u,v\in E_n^{(m)}$. Write
\[
u=\frac{A}{n^m},
\qquad
v=\frac{B}{n^m}.
\]
By the induction hypothesis, $B\equiv0$ or $1\pmod{n-1}$. The numerator of
$z$ at level $m+1$ is
\[
K=(n-1)A+B\equiv B\pmod{n-1}.
\]
Hence $K\equiv0$ or $1\pmod{n-1}$.
\end{proof}

\begin{thm}\label{thm:En-proper}
If $n\ge4$, then $2/n\notin E_n$. Consequently,
$E_n\subsetneq D_n$ for $n\ge4$.
\end{thm}

\begin{proof}
Assume $2/n\in E_n$. By Lemma~\ref{lem:En-nesting}, there is $M\ge1$ such
that $2/n\in E_n^{(M)}$. Writing
\[
\frac{2}{n}=\frac{2n^{M-1}}{n^M},
\]
Lemma~\ref{lem:residue-invariant} gives
\[
2n^{M-1}\equiv0\quad\text{or}\quad1\pmod{n-1}.
\]
Since $n\equiv1\pmod{n-1}$, this implies
\[
2\equiv0\quad\text{or}\quad1\pmod{n-1},
\]
which is impossible for $n\ge4$.
\end{proof}

\begin{rem}\label{rem:role-En}
The preceding results show that the distinction between $D_n$ and $E_n$ is
invisible in the binary and ternary cases, but becomes genuine for $n\ge4$.
This is why the full $D_n$-construction is needed in the general proof, even
though the eventual continuity mechanism only uses the $E_n$-type two-point
averages $((n-1)u+v)/n$.
\end{rem}

\section{The general \texorpdfstring{$D_n$}{Dn}-construction}
\label{sec:Dn-construction}

Let $I=[a,b]$ be a non-degenerate compact interval and let
$F:I^n\to I$ be reflexive, symmetric, bisymmetric and partially strictly
increasing. We construct a strictly increasing function $f:D_n\to I$
satisfying the full averaging identity on $D_n$.

\subsection{Recursive definition and monotonicity}

Set
\[
f(0)=a,
\qquad
f(1)=b.
\]
Assume that $f$ is already defined on $D_{n,m}$. We keep these old
values. For $k\in\{0,\dots,n^m-1\}$ and
$\delta\in\{1,\dots,n-1\}$, define
\begin{equation}\label{eq:Dn-def}
f\!\left(\frac{kn+\delta}{n^{m+1}}\right)=
F\Bigl(
\underbrace{f\!\left(\frac{k}{n^m}\right),\dots,
f\!\left(\frac{k}{n^m}\right)}_{n-\delta},
\underbrace{f\!\left(\frac{k+1}{n^m}\right),\dots,
f\!\left(\frac{k+1}{n^m}\right)}_{\delta}
\Bigr).
\end{equation}
Equivalently, the same formula also holds for $\delta=0$, in which case it
just reproduces the old value by reflexivity. The endpoint $1$ keeps its
previous value $f(1)=b$..

\begin{prop}\label{prop:Dn-monotone}
The recursively defined function $f:D_n\to I$ is strictly increasing.
\end{prop}

\begin{proof}
The proof is the same induction as in the ternary case.

We use the convention that the endpoint $1\in D_{n,m+1}$ is represented as
\[
1=\frac{n^m n+0}{n^{m+1}},
\]
and in this case the upper block in \eqref{eq:Dn-def} is empty.

Suppose $f$ is strictly increasing on $D_{n,m}$ and
let
\[
x=\frac{pn+\delta}{n^{m+1}}<\frac{qn+\epsilon}{n^{m+1}}=y.
\]

If $p=q$, then $\delta<\epsilon$ and the defining expression for $f(y)$ is obtained from that for $f(x)$ by
replacing at least one lower endpoint value by the strictly larger upper endpoint value.

If $p<q$, then
\[
\frac{p}{n^m}<\frac{q}{n^m},
\qquad
\frac{p+1}{n^m}\le\frac{q}{n^m},
\]
so again the entries defining $f(x)$ are coordinatewise no larger than those defining $f(y)$, with at least one
strict inequality. Partial strict increase of $F$ gives $f(x)<f(y)$.
\end{proof}

\subsection{The averaging identity on \texorpdfstring{$D_n$}{Dn}}

\begin{thm}\label{thm:Dn-identity}
For all $x_1,\dots,x_n\in D_n$,
\[
F(f(x_1),\dots,f(x_n))=
f\!\left(\frac{x_1+\cdots+x_n}{n}\right).
\]
\end{thm}

\begin{proof}
We prove the assertion by induction on $m$. More precisely, we prove that
the identity holds whenever $x_1,\dots,x_n\in D_{n,m}$.

For $m=0$, each $x_i$ is either $0$ or $1$. If exactly $\delta$ of
the numbers $x_1,\dots,x_n$ are equal to $1$, then, by symmetry of $F$
and by the definition of $f$,
\[
F(f(x_1),\dots,f(x_n))
=
F(\underbrace{a,\dots,a}_{n-\delta},
  \underbrace{b,\dots,b}_{\delta})
=
f\left(\frac{\delta}{n}\right)
=
f\!\left(\frac{x_1+\cdots+x_n}{n}\right).
\]
Thus the assertion holds for $m=0$.

Assume now that the assertion is already known on $D_{n,m}$, and let
$x_1,\dots,x_n\in D_{n,m+1}$. Write
\[
x_i=\frac{k_i n+\delta_i}{n^{m+1}},
\qquad
0\le \delta_i\le n-1,
\]
where $0\le k_i\le n^m$, and where $k_i=n^m$ can occur only when
$\delta_i=0$, that is, only for the endpoint $x_i=1$.
Whenever $\delta_i=0$, the upper block below is empty and no value
$f((k_i+1)/n^m)$ is involved.
Then, by the recursive definition of $f$,
\begin{multline*}
F(f(x_1),\dots,f(x_n))  \\
=
F\Biggl(
F\left(
\underbrace{f\!\left(\frac{k_1}{n^m}\right),\dots,
f\!\left(\frac{k_1}{n^m}\right)}_{n-\delta_1},
\underbrace{f\!\left(\frac{k_1+1}{n^m}\right),\dots,
f\!\left(\frac{k_1+1}{n^m}\right)}_{\delta_1}
\right),
\\
\ldots,
\\
F\left(
\underbrace{f\!\left(\frac{k_n}{n^m}\right),\dots,
f\!\left(\frac{k_n}{n^m}\right)}_{n-\delta_n},
\underbrace{f\!\left(\frac{k_n+1}{n^m}\right),\dots,
f\!\left(\frac{k_n+1}{n^m}\right)}_{\delta_n}
\right)
\Biggr).
\end{multline*}

Set
\[
\delta_1+\cdots+\delta_n=qn+\delta,
\qquad
q\in\{0,\dots,n-1\},\quad \delta\in\{0,\dots,n-1\},
\]
and
\[
K=k_1+\cdots+k_n+q.
\]

We now use symmetry of $F$ inside each of the inner $F$'s. In the first
row we put the $+1$-elements in positions
\[
1,\dots,\delta_1.
\]
In the second row we put the $+1$-elements immediately after these positions,
cyclically modulo $n$. In general, if the first
$\delta_1+\cdots+\delta_{i-1}$ positions have already been used, then in the
$i$-th row we put the $\delta_i$ many $+1$-elements in the next
$\delta_i$ positions, again cyclically modulo $n$.

Since $\delta_i\le n-1$, no position is repeated inside a single row, so this
is indeed only a rearrangement of the arguments in that row. Hence the value
of the whole expression is unchanged by symmetry.

After all rows have been rearranged in this way, the total number of
$+1$-elements is
\[
\delta_1+\cdots+\delta_n=qn+\delta.
\]
Therefore, among the $n$ columns, exactly $\delta$ columns contain
$q+1$ many $+1$-elements, and the remaining $n-\delta$ columns contain
$q$ many $+1$-elements.

Now we apply bisymmetry. This allows us to transpose the resulting
$n\times n$ array of arguments: instead of first applying $F$ along the
rows and then applying $F$ to the obtained $n$ values, we may first apply
$F$ along the columns and then apply $F$ to the obtained $n$ values.

Consider one of the columns. If this column contains $q+1$ many
$+1$-elements, then, by the induction hypothesis applied to the $n$ points
of $D_{n,m}$ appearing in this column, the value of the corresponding inner
$F$ is
\[
f\!\left(
\frac{k_1+\cdots+k_n+q+1}{n^{m+1}}
\right)
=
f\!\left(\frac{K+1}{n^{m+1}}\right).
\]
Similarly, if the column contains $q$ many $+1$-elements, then the value of
the corresponding inner $F$ is
\[
f\!\left(
\frac{k_1+\cdots+k_n+q}{n^{m+1}}
\right)
=
f\!\left(\frac{K}{n^{m+1}}\right).
\]

Thus, after the application of bisymmetry and of the induction hypothesis, we
obtain
\begin{multline*}
F(f(x_1),\dots,f(x_n))
\\
=
F\left(
\underbrace{
f\!\left(\frac{K+1}{n^{m+1}}\right),\dots,
f\!\left(\frac{K+1}{n^{m+1}}\right)
}_{\delta},
\underbrace{
f\!\left(\frac{K}{n^{m+1}}\right),\dots,
f\!\left(\frac{K}{n^{m+1}}\right)
}_{n-\delta}
\right).
\end{multline*}
By symmetry we may place the $n-\delta$ copies of
$f(K/n^{m+1})$ first and the $\delta$ copies of
$f((K+1)/n^{m+1})$ last. Hence, by the recursive definition of $f$,
\[
F(f(x_1),\dots,f(x_n))
=
f\!\left(\frac{Kn+\delta}{n^{m+2}}\right).
\]
Finally, since
\[
K=k_1+\cdots+k_n+q
\]
and
\[
\delta_1+\cdots+\delta_n=qn+\delta,
\]
we have
\[
Kn+\delta
=
n(k_1+\cdots+k_n)+qn+\delta
=
n(k_1+\cdots+k_n)+(\delta_1+\cdots+\delta_n).
\]
Therefore
\[
\frac{Kn+\delta}{n^{m+2}}
=
\frac{(k_1n+\delta_1)+\cdots+(k_n n+\delta_n)}{n^{m+2}}
=
\frac{x_1+\cdots+x_n}{n}.
\]
Consequently,
\[
F(f(x_1),\dots,f(x_n))
=
f\!\left(\frac{x_1+\cdots+x_n}{n}\right),
\]
which completes the induction step.
\end{proof}

\section{A dense-domain continuity lemma}
\label{sec:dense-lemma}

For weights $w_1,\dots,w_n>0$ with $\sum_i w_i=1$, write
\[
m_w(x_1,\dots,x_n)=\sum_{i=1}^n w_i x_i.
\]

\begin{prop}[Continuity and representation from a dense domain]\label{prop:dense-domain}
Let $I=[a,b]$ be a non-degenerate compact interval and let
$F:I^n\to I$ be partially strictly increasing. Let $S\subset[0,1]$ be dense,
contain $0$ and $1$, and be closed under $m_w$. Suppose that
$f:S\to I$ is strictly increasing, $f(0)=a$, $f(1)=b$, and
\begin{equation}\label{eq:dense-identity}
F(f(x_1),\dots,f(x_n))=f(m_w(x_1,\dots,x_n))
\qquad (x_i\in S).
\end{equation}
Then $f(S)$ is dense in $I$. Consequently $f$ extends uniquely to a
continuous strictly increasing bijection
\[
\widetilde f:[0,1]\to I,
\]
and
\begin{equation}\label{eq:dense-representation}
F(u_1,\dots,u_n)=
\widetilde f\!\left(\sum_{i=1}^n w_i\widetilde f^{-1}(u_i)\right)
\qquad (u_i\in I).
\end{equation}
In particular, $F$ is continuous.
\end{prop}

\begin{proof}
Define
\[
\widehat f(x)=\sup\{f(s):s\in S,\ s\le x\}\qquad (x\in[0,1]).
\]
Then $\widehat f$ is an increasing extension of $f$. It is strictly
increasing: if $x<y$, choose $s,t\in S$ with $x<s<t<y$; then
\[
\widehat f(x)\le f(s)<f(t)\le\widehat f(y).
\]
Hence $\overline{f(S)}$ is uncountable.

A compact subset of the real line has at most countably many isolated points and
at most countably many one-sided accumulation points. Therefore
$\overline{f(S)}$ has uncountably many two-sided accumulation points.

Assume, for contradiction, that $f(S)$ is not dense in $I$. Since $f$ is
strictly increasing and $S$ is dense, this gives a cut $z\in(0,1)$ such that
\[
X:=\sup\{f(s):s\in S,\ s<z\}
<
Y:=\inf\{f(s):s\in S,\ s>z\}.
\]
Let $\alpha<\beta$ be two distinct two-sided accumulation points of
$\overline{f(S)}$. Choose $p,q\in S$ with
\[
\alpha<f(p)<f(q)<\beta.
\]
Choose $d,D\in S$ with $d<z<D$ so close to $z$ that
\[
f(d)<X<Y<f(D)
\]
and
\[
D-d<\frac{w_n}{1-w_n}(q-p).
\]
Then
\[
A_1:=(1-w_n)D+w_np,
\qquad
A_2:=(1-w_n)d+w_nq
\]
belong to $S$ and satisfy $A_1<A_2$. Thus $f(A_1)<f(A_2)$, and by
\eqref{eq:dense-identity},
\[
F(\underbrace{f(D),\dots,f(D)}_{n-1},f(p))
<
F(\underbrace{f(d),\dots,f(d)}_{n-1},f(q)).
\]
Together with partial strict increase this yields
\begin{align*}
F(\underbrace{f(d),\dots,f(d)}_{n-1},f(p))
&<F(\underbrace{f(D),\dots,f(D)}_{n-1},f(p))\\
&<F(\underbrace{f(d),\dots,f(d)}_{n-1},f(q))\\
&<F(\underbrace{f(D),\dots,f(D)}_{n-1},f(q)).
\end{align*}
Since $f(d)<X<Y<f(D)$, we get
\[
F(\underbrace{Y,\dots,Y}_{n-1},f(p))
<
F(\underbrace{X,\dots,X}_{n-1},f(q)).
\]
Using $\alpha<f(p)<f(q)<\beta$, this implies
\[
F(\underbrace{Y,\dots,Y}_{n-1},\alpha)
<
F(\underbrace{X,\dots,X}_{n-1},\beta).
\]
The above argument applies to every pair $\alpha<\beta$ of two-sided accumulation
points. Hence, if $t_1<t_2$ are two such points, then
\[
F(\underbrace{Y,\dots,Y}_{n-1},t_1)
<
F(\underbrace{X,\dots,X}_{n-1},t_2).
\]
Therefore, the intervals
\[
\Phi_t:=\left(
F(\underbrace{X,\dots,X}_{n-1},t),
F(\underbrace{Y,\dots,Y}_{n-1},t)
\right)
\]
are pairwise disjoint as $t$ runs through the set of two-sided accumulation
points of $\overline{f(S)}$. Each $\Phi_t$ is non-empty because $X<Y$ and
$F$ is strictly increasing in the first $n-1$ variables. This gives
uncountably many pairwise disjoint non-empty open intervals in $I$, impossible.
Hence $f(S)$ is dense in $I$.

Since $f(S)$ is dense, the monotone extension $\widehat f$ has no jumps; a
jump would create a non-empty open interval missing from $f(S)$. Hence
$\widehat f$ is continuous. Set $\widetilde f=\widehat f$. Then
$\widetilde f$ is the unique continuous strictly increasing bijective extension
of $f$.

It remains to extend the identity. Let $u_i\in I$ and set
$t_i=\widetilde f^{-1}(u_i)$. Choose sequences
$r_i^{(m)},s_i^{(m)}\in S$ with
\[
r_i^{(m)}\le t_i\le s_i^{(m)},
\qquad
r_i^{(m)}\to t_i,
\qquad
s_i^{(m)}\to t_i.
\]
Then
\[
f(r_i^{(m)})\le u_i\le f(s_i^{(m)}).
\]
By partial monotonicity and \eqref{eq:dense-identity},
\[
\widetilde f(m_w(r_1^{(m)},\dots,r_n^{(m)}))
\le F(u_1,\dots,u_n)
\le
\widetilde f(m_w(s_1^{(m)},\dots,s_n^{(m)})).
\]
Letting $m\to\infty$ and using continuity of $\widetilde f$ gives
\eqref{eq:dense-representation}.
\end{proof}

\begin{rem}\label{rem:En-in-dense}
The proof of Proposition~\ref{prop:dense-domain} uses only barycenters of the
form
\[
(1-w_n)u+w_nv.
\]
In the equal-weight case $w_i=1/n$, these are exactly the $E_n$-type points
\[
\frac{(n-1)u+v}{n}.
\]
This is the precise sense in which $E_n$ is the natural control set for
continuity part of the argument.
\end{rem}

\section{Proof of the main theorem}
\label{sec:proof-main}

\begin{proof}[Proof of Theorem~\ref{thm:main}]
By Lemma~\ref{lem:compact-reduction}, it is enough to prove the theorem on
compact intervals. Thus let $I=[a,b]$ be a non-degenerate compact interval.

By Theorem~\ref{thm:Dn-identity}, the recursively constructed function
$f:D_n\to I$ satisfies
\[
F(f(x_1),\dots,f(x_n))
=
f\!\left(\frac{x_1+\cdots+x_n}{n}\right)
\qquad (x_i\in D_n).
\]
By Proposition~\ref{prop:Dn-monotone}, $f$ is strictly increasing, and by
construction $f(0)=a$, $f(1)=b$. By Lemma~\ref{lem:Dn-dense}, $D_n$ is
dense in $[0,1]$ and closed under $n$-fold arithmetic averaging. Thus
Proposition~\ref{prop:dense-domain}, applied with $S=D_n$ and equal weights
$w_i=1/n$, gives a continuous strictly increasing bijection
$\widetilde f:[0,1]\to I$ such that
\[
F(u_1,\dots,u_n)=
\widetilde f\!\left(
\frac{\widetilde f^{-1}(u_1)+\cdots+\widetilde f^{-1}(u_n)}{n}
\right).
\]
Equivalently, with
\[
\varphi=\widetilde f^{-1}:I\to[0,1],
\]
we have
\[
F(u_1,\dots,u_n)=
\varphi^{-1}\!\left(
\frac{\varphi(u_1)+\cdots+\varphi(u_n)}{n}
\right).
\]
Thus $F$ is continuous and has the required quasi-arithmetic representation
on compact intervals. The general interval case follows from
Lemma~\ref{lem:compact-reduction}.
\end{proof}

\section{Application: Kolmogorov means without a continuity assumption}
\label{sec:kolmogorov}

We now show how the same dense-domain argument removes the continuity assumption
from the Kolmogorov--Nagumo--de Finetti compatible-family theorem, provided
strict monotonicity is assumed. More precisely, we prove the following compact
form of Theorem~\ref{thm:KND-without-continuity}; the general interval case
then follows by the same compact-reduction argument as before, and the two
representations are equivalent after writing $\varphi=\psi^{-1}$.

Let $I=[a,b]$ be a non-degenerate compact interval. Consider a family of maps
\[
M_n:I^n\to I\qquad(n\ge1),
\]
with $M_1(x)=x$. We assume that the family satisfies the replacement axiom
from Definition~\ref{def:replacement}; explicitly, for all $k,m\ge1$,
\begin{multline}\label{eq:Kolmogorov-axiom}
M_{k+m}(x_1,\dots,x_k,y_1,\dots,y_m)\\
=
M_{k+m}\bigl(
\underbrace{M_k(x_1,\dots,x_k),\dots,M_k(x_1,\dots,x_k)}_{k},
y_1,\dots,y_m
\bigr).
\end{multline}
By symmetry, the same replacement can be applied to any block of variables.

\begin{thm}[A Kolmogorov-type theorem without a continuity assumption]\label{thm:kolmogorov}
Assume that every $M_n$ is reflexive, symmetric and partially strictly increasing in
each variable, and that \eqref{eq:Kolmogorov-axiom} holds. Then there exists a
continuous strictly increasing bijection $\psi:[0,1]\to I$ such that, for
every $n\ge1$,
\[
M_n(x_1,\dots,x_n)=
\psi\!\left(\frac{\psi^{-1}(x_1)+\cdots+\psi^{-1}(x_n)}{n}\right)
\qquad (x_i\in I).
\]
In particular, every $M_n$ is continuous.
\end{thm}

\begin{proof}
First we prove a replication identity. For every $p,m\ge1$ and
$x_1,\dots,x_m\in I$,
\begin{equation}\label{eq:replication}
M_{pm}(\underbrace{x_1,\dots,x_1}_{p},\dots,
\underbrace{x_m,\dots,x_m}_{p})
=
M_m(x_1,\dots,x_m).
\end{equation}
Indeed, by symmetry we may arrange the variables into $p$ identical blocks
$(x_1,\dots,x_m)$. Applying the replacement axiom to each block gives $pm$
copies of $M_m(x_1,\dots,x_m)$, and reflexivity then gives
\eqref{eq:replication}.

Let $S=\Q\cap[0,1]$. For $0\le p\le q$, $q\ge1$, define
\begin{equation}\label{eq:def-psi-K}
\psi\!\left(\frac pq\right)
=
M_q(\underbrace{a,\dots,a}_{q-p},\underbrace{b,\dots,b}_{p}).
\end{equation}
The endpoint cases $p=0$ and $p=q$ are immediate. If $p/q=p'/q'$, then by
\eqref{eq:replication},
\begin{align*}
M_q((q-p)a,pb)&=M_{qq'}(q'(q-p)a,q'pb),\\
M_{q'}((q'-p')a,p'b)&=M_{qq'}(q(q'-p')a,qp'b),
\end{align*}
where $M_r(\alpha a,\beta b)$ denotes $M_r$ applied to $\alpha$ copies of
$a$ and $\beta$ copies of $b$. Since $pq'=p'q$, the numbers of $a$'s
and $b$'s are the same. Thus $\psi$ is well-defined.

The function $\psi$ is strictly increasing on $S$. Indeed, if $p<p'$ and
the denominator is common, then
\[
\psi(p/q)=M_q((q-p')a,(p'-p)a,pb)
<
M_q((q-p')a,(p'-p)b,pb)=\psi(p'/q),
\]
using $a<b$ and strict increase.

We next prove the rational representation. Let $z_i=p_i/q\in S$ with a common
denominator $q$. Put
\[
B_i=((q-p_i)a,p_i b),
\]
a block of length $q$. Starting from $M_{nq}(B_1,\dots,B_n)$, apply the
replacement axiom to each block $B_i$, and then use
\eqref{eq:replication}; this gives
\[
M_n(M_q(B_1),\dots,M_q(B_n))=M_{nq}(B_1,\dots,B_n).
\]
Therefore
\begin{align*}
M_n(\psi(z_1),\dots,\psi(z_n))
&=M_{nq}((q-p_1)a,p_1b,\dots,(q-p_n)a,p_nb)\\
&=M_{nq}((nq-p_1-\cdots-p_n)a,(p_1+\cdots+p_n)b)\\
&=\psi\!\left(\frac{p_1+\cdots+p_n}{nq}\right)
=\psi\!\left(\frac{z_1+\cdots+z_n}{n}\right).
\end{align*}

Now fix $n\ge2$ and apply Proposition~\ref{prop:dense-domain} to
$F=M_n$, $S=\Q\cap[0,1]$, $f=\psi$, and
$w_1=\cdots=w_n=1/n$. The set $S$ is dense and closed under arithmetic
averages, and the identity just proved is exactly \eqref{eq:dense-identity}.
Hence $\psi$ extends uniquely to a continuous strictly increasing bijection
$\widetilde\psi:[0,1]\to I$, and the displayed representation holds for this
fixed $n$. The extension is independent of $n$, since it is the unique
continuous extension of the same strictly increasing function $\psi$ on the
dense set $S$. Thus the formula holds for every $n\ge2$, while the case
$n=1$ is $M_1=\id$. Replacing $\widetilde\psi$ by $\psi$ finishes the
proof.
\end{proof}

\section{Concluding remarks and open problems}
\label{sec:open}

The proof of Theorem~\ref{thm:main} shows that, in the symmetric reflexive case,
bisymmetry is strong enough to propagate a purely order-theoretic coding from a
dense rational-like set to the whole interval. The dense-domain lemma is the
regularity mechanism: once an averaging identity holds on a dense set, any gap
in the range of the coding function would generate uncountably many disjoint
intervals.

The main remaining question is whether symmetry is necessary.

\begin{open}\label{open:nonsymmetric}
Let $I=[a,b]$ be a non-degenerate compact interval and let $F:I^n\to I$ be
reflexive, bisymmetric and partially strictly increasing. Is $F$ necessarily
continuous? Equivalently, must such an operation have a weighted
quasi-arithmetic form?
\end{open}

This problem is open already in the binary case. There is, however, a partial
structural result in this direction: Burai, Kiss and Szokol \cite{BuraiKissSzokol2023} proved a dichotomy
for strictly increasing bisymmetric binary maps. In the reflexive setting, such
an operation is either symmetric everywhere or nowhere symmetric. In the
symmetric case, the automatic-continuity theorem of Burai--Kiss--Szokol applies;
the genuinely non-symmetric case remains open.

It is plausible that the binary problem plays a decisive role. Indeed, several
continuous representation theorems for $n$-variable bisymmetric means,
including Maksa's approach, reduce the $n$-variable situation to the binary
Aczel theorem. A comparable reduction without assuming continuity would be very
useful, but we do not pursue it here.

Reflexivity, on the other hand, cannot simply be dropped. In \cite{Kiss2026},
discontinuous bisymmetric strictly increasing operations are constructed on real
intervals in the non-reflexive regime, including operations of the form
\[
F(x_1,\dots,x_n)=f^{-1}\!\left(\sum_{i=1}^n \alpha_i f(x_i)\right),
\qquad
\alpha_i>0,
\qquad
\sum_i\alpha_i\ne1,
\]
with a strictly increasing but discontinuity-producing coding through a nowhere
dense set. Thus Open Problem~\ref{open:nonsymmetric} is the natural reflexive
boundary case.

\section*{Acknowledgements} The first author was supported by the Hungarian National Research, Development
and Innovation Office grants K146922, STARTING 150576 and FK 142993.

\noindent
{\sc Gergely Kiss:}\\
Department of Mathematics, Corvinus University of Budapest, \\
Fővám tér 13-15, Budapest HU-1093, Hungary,\\
and\\
HUN-REN Alfred Renyi Institute of Mathematics\\
Re\'altanoda utca 13-15, H-1053, Budapest, Hungary\\
E-mail: {\tt kiss.gergely@renyi.hu, kigergo57@gmail.com}

\bigskip

\noindent
{\sc Ekaterina Shulman:}\\
Institute of Mathematics, University of Silesia,  \\
14 Bankowa Str., Katowice PL-40-007, Poland  \\
Email:    {\tt ekaterina.shulman@us.edu.pl,   shulmanka@gmail.com}

\end{document}